\documentclass[a4paper,11pt]{amsart}
\usepackage{hyperref,latexsym}
\usepackage{enumerate}

\theoremstyle{plain}
\newtheorem{theorem}{Theorem}[section]

\theoremstyle{definition}

\newtheorem{example}{Example}
\theoremstyle{remark}

\begin{document}

\title[Total integrability at infinity]
      {On Newton equations which are totally integrable at infinity}

\date{20 May 2015}
\author{Misha Bialy}
\address{School of Mathematical Sciences, Raymond and Beverly Sackler Faculty of Exact Sciences, Tel Aviv University,
Israel} \email{bialy@post.tau.ac.il}


\begin{abstract} In this paper Hamiltonian system of time dependent periodic
Newton equations is studied. It is shown that for dimensions $3$ and
higher the following rigidity results holds true: If all the orbits
in a neighborhood of infinity are action minimizing then the
potential must be constant. This gives a generalization of the
previous result \cite{B3}, where it was required all the orbits to
be minimal. As a result we have the following application: Suppose
that for the time-1 map of the Hamiltonian flow there exists a
neighborhood of infinity which is filled by invariant Lagrangian
tori homologous to the zero section. Then the potential must be
constant. Remarkably, the statement is false for $n=1$ case and
remains unknown to the author for $n=2$.
\end{abstract}

\maketitle

\section{Introduction and the result}
In this note we consider the system of Newton equations (\ref{eq1})
with time dependent periodic potential $u(q,t)$:
$$u:\mathbf{T}^n\times\mathbf{S}^1\rightarrow\mathbf{R}.
$$ Here $\mathbf{T}^n=\mathbf{R}^n/\Gamma$ is an $n$-torus, $\mathbf{S}^1=\mathbf{R}/\mathbf{Z}$
 and
$q=(q_1,..,q_n)$ and $t$ stand for standard coordinates on
$\mathbf{R}^n$ and $\mathbf{R}$ respectively. We fix a Riemannian
metric $g$ on the torus and study the system of Newton equations
determined by $g$:
\begin{equation}
\label{eq1} \nabla_{\dot{q}}\dot{q}=-\nabla_q u
\end{equation}
(here and later $\nabla, |\cdot|$ are computed in terms of the
Riemannian metric $g$). Classically, the extremals of the action
functional
$$\int \left(\frac{1}{2}|\dot{q}|^2-u(q,t)\right)dt$$
coincide with the solutions of (\ref{eq1}) and are described in
Hamiltonian formalism by means of the Hamiltonian flow of
\begin{equation}
\label{eq2} H: T^*\mathbf{T}^n\times\mathbf{S}^1\rightarrow
\mathbf{R},\quad H(p,q,t)=\frac{1}{2}|p|^2+u(q,t). \end{equation}
Let $\phi$ denotes the time-1 map of the Hamiltonian flow of $H$.

The main object of our study is the set
$$\mathcal{M}\subseteq T^*\mathbf{T}^n\times\mathbf{S}^1$$ swept by
those infinite orbits of the Hamiltonian flow $(p(t),q(t),t)$ such
that the corresponding extremals $(q(t))$ have no conjugate points,
or, in other words which are local minima of the action functional
between any two of its points. By definition, $\mathcal{M}$ is a
closed invariant subset of the phase space $
T^*\mathbf{T}^n\times\mathbf{S}^1$.

The purpose of this note is to prove the following generalization of
rigidity theorem discovered in \cite{B3}:
\begin{theorem}
\label{main} Let $n\geq 3$. Assume that the set $\mathcal{M}$
contains a neighborhood of infinity of
$T^*\mathbf{T}^n\times\mathbf{S}^1$:
$$\mathcal{M}\supset \{|p|>R\},
$$
for some positive constant $R$. Then
\begin {enumerate}
\item The metric $g$ must be Euclidean.
\item The potential $u$ does not depend on $q$.
\end{enumerate}
\end{theorem}

Theorem \ref{main} implies the following corollary for phase
portraits of integrable Hamiltonians. We shall say that the system
(\ref{eq1}) is \textit{totally integrable at
 infinity} if there exists a neighborhood of infinity of
 $T^*\mathbf{T}^n$ filled by Lagrangian tori homologous to
 the zero section which stay invariant under the time-1 map $\phi$.
\begin{theorem}
\label{cor} Let $n\geq 3$, and suppose that the Hamiltonian system
(\ref{eq2}) is \textit{totally integrable at infinity}. Then the
Riemannian metric $g$ must be Euclidean and the potential $u$ does
not depend on $q$.
\end{theorem}

Theorem \ref{cor} follows from Theorem \ref{main} applying the so
called generalized Birkhoff theorem. This theorem states that every
Lagrangian torus homologous to the zero section which is invariant
under $\phi$ is a graph, and therefore consists of minimal orbits.
The generalized Birkhoff theorem was proved in series of joint
papers with L.Polterovich (see \cite{BP} for the case most suitable
for this paper) where an extra dynamical assumption was imposed.
Nowadays, it is known to be true without this assumption due to the
final solution obtained in \cite{arnaud} by a different methods.


Let me point out that for Euclidean $g$ and arbitrary periodic
potential $u$, by KAM theory, there are lots of invariant tori of
the Hamiltonian system (\ref{eq2}) lying near $\infty$. All of them
consist of minimal orbits.

The statements of Theorems\ref{main},\ref{cor} in the case $n=1$ are
false as the following example shows.

\begin{example} Let $n=1$. Consider any non-constant autonomous periodic potential
$u(q)$ so that $H(p,q,t)=\frac{1}{2}p^2+u(q)$. As the energy
constant $h$ varies the energy level curves $\{H=h\}$ undergo
\textit{perestroika} on the phase cylinder and it is easy to see
that
$$\mathcal{M}=\mathop{\bigcup}_{h\geq \max{u}}\left\{p=\pm\sqrt{2(h-u)}\right\}.$$
Therefore, for $n=1$ any autonomous potential $u$ satisfies the
assumption of the theorem. Another example is the system (\ref{eq2})
with periodic potential of the form $u=u(mq+nt),$ which is
integrable and has a similar phase portrait as above.
\end{example}
\newpage
\textbf{Questions and Remarks.}

1) In \cite{B3} the case
$\mathcal{M}=T^*\mathbf{T}^n\times\mathbf{S}^1$ is considered for
any $n\geq 1$. Let me remark that for $n=2$ the method described
below does not work so it is unclear if the result of the theorems
remains valid in this case. Obviously, one can not create
counter-examples by taking direct products of known Hamiltonians.

2) It is an interesting question if the statement allows further
refinement. For instance does the result hold true assuming that the
complement of the set $\mathcal{M}$ has finite measure?

3) Another open question for the case $n=1$: is it true that the
systems described in the example are the only ones which are totally
integrable at infinity? Are there integrable systems (\ref{eq2})
which are not totally integrable at infinity?

The strategy of the proofs of Theorem \ref{main} is similar to that
of \cite{B3}. In this paper in order to handle the non-compactness
of the phase space we need to introduce suitably chosen cutoff
function in the phase space and then to stretch it to become
concentrated closer and closer to infinity.

The paper is organized as follows: in Section 2 we use Burago-Ivanov
theorem to prove that the metric $g$ has to be Euclidean; in Section
3 we get an inequality applying Hopf method, and finally in section
4 we use the stretching of the cutoff function to prove the reverse
inequality unless $\nabla_q u$  vanishes identically. Combination of
the results of these sections yields the proof.
\section {Flatness of the metric $g$}
In this section we show that the argument of \cite{B3} applies in
our situation. So assume that $\mathcal{M}$ contains a neighborhood
of infinity. We have to show that the metric $g$ is Euclidean. In
the opposite case it follows from \cite{BI} that there are geodesics
of $g$ which have conjugate points. Let $q(t),t\in[0;T]$ be a
segment of such a geodesic so that $q(0)$ and $q(T_*)$ for some
$0<T_*<T$ are conjugate. Let $\gamma=(p(t),q(t))$ be the
corresponding orbit of the Hamiltonian flow of the metric $g$,
$|p(t)|=1$. Consider now the orbit
$\tilde{\gamma}=(\tilde{p}(t),\tilde{q}(t)),t\in[0;T]$ of the
Hamiltonian flow with the perturbed function
$$H_{\epsilon}=\frac{1}{2}|p|^2+{\epsilon}^2u(q,\epsilon t)$$ with
the same initial conditions as $\gamma$, and for $\epsilon$ small
enough. Then by continuous dependence, the extremal $(\tilde{q}(t))$
also has conjugate points somewhere on the segment $[0;T]$.

Moreover notice, that there is a correspondence between the orbits
$(p(t),q(t))$ of $H$ and $(\epsilon p(\epsilon t),q(\epsilon t))$ of
$H_{\epsilon}$ for any $\epsilon$.

Using this correspondence we get the orbit
$\Gamma=(\frac{1}{\epsilon} \tilde{p}(\frac{1}{\epsilon}t),
\tilde{q}(\frac{1}{\epsilon}t)), t\in[0;T]$ of $H$ also has
conjugate points. The last thing is to see that we started with a
geodesic segment on the energy level $|p|=1$, so for $\epsilon$
small enough the constructed segment $\Gamma$ lies in
$\{|p|>R\}\subset\mathcal{M}$. This contradiction proves the claim.

\section {E.Hopf method and cutoff function}
From now on we shall assume that the Riemannian metric $g$ is
standard Euclidean. This is in fact the main case of this note.

 First of all notice that since only gradient of
the potential $u$ is involved in the equations of the system
(\ref{eq1}) we are free to add any function of $t$ to $u$, so we
shall assume everywhere in the sequel that for some constant $M>0$:
$$
0\leq u(q,t)\leq M,\ \forall (q,t).
$$
It then follows that high energy levels of $H$ lie entirely in
$\mathcal{M}$:

\begin{equation}
\label{rho} \forall h>\frac{R^2}{2}+M \Longrightarrow \{H=h\}\subset
\{|p|>R\}\subset \mathcal{M}.
\end{equation}

We shall denote by $\rho$ a non-negative smooth function of one
variable (which will be composed with the Hamiltonian $H$ later on)
with compact support so that
$$
supp(\rho) \subset \left(\frac{R^2}{2}+M;+\infty \right).
$$
Let me denote by $\mu$ the invariant measure $d\mu=dpdqdt$. We have:
\begin{theorem}
\label{D} Suppose that the set $\mathcal{M}$ contains the
neighborhood of infinity $\{|p|>R\}$. Then for any function $\rho$
defined above one has the inequality:
$$D=\int(\rho'(H))^2u_t^2d\mu+
\frac{1}{n}\int(\rho^2)'(H)(\nabla_q u)^2d\mu\geq0.
$$
\end{theorem}

\begin{proof}
We proceed as in the original Hopf method \cite{H},\cite{green} and
also \cite{knauf},\cite{B3} and construct a measurable matrix
function
$$A:\mathcal{M}\rightarrow\mathbf{R},$$
satisfying the matrix Riccati equation:
$$L_vA+A^2+Hess(u)=0,$$
where $L_v$ denotes the Lie derivative along the vector field
$$v=\partial_t+\sum\limits_{i=1}^n(p_i\partial_{q_i}-u_{q_i}\partial_{p_i}).$$ In addition $A$ and $L_vA$
are uniformly bounded on the whole $\mathcal{M}$. Denoting $a=Tr A$
and using the inequality for the trace $Tr A^2\geq\frac{1}{n}(Tr
A)^2,$ we have the following inequality for $a$:
\begin{equation}
\label{eq3} L_va+\frac{1}{n}a^2+\Delta_q u\leq0.
\end{equation}

Next, we multiply the inequality (\ref{eq3}) by the function
$\rho^2\circ H$:
$$
\rho^2(H)L_v a+\frac{1}{n}\rho^2(H)a^2+\rho^2(H)\Delta_q u\leq0.
$$
Or equivalently:
$$
L_v(\rho^2(H) a)-aL_v(\rho^2(H))
+\frac{1}{n}\rho^2(H)a^2+\rho^2(H)\Delta_q u\leq0.
$$
Since $L_vH=u_t$ we have
\begin{equation}
\label{eq4} L_v(\rho^2(H) a)-2a\rho'(H)\rho(H)u_t
+\frac{1}{n}\rho^2(H)a^2+\rho^2(H)\Delta_q u\leq0.
\end{equation}
Then we integrate this inequality over the invariant set
$\mathcal{M}$ with respect to invariant measure $d\mu=dpdqdt$. Since
the support of $\rho\circ H$ lies entirely in $\mathcal{M}$, by
(\ref{rho}), we can extend the integration to the whole
$T^*\mathbf{T}^n\times\mathbf{S}^1$:
\begin{equation}
\label{eq6} -2\int a\rho'(H)\rho(H)u_td\mu
+\frac{1}{n}\int\rho^2(H)a^2d\mu+ \int\rho^2(H)\Delta_q ud\mu\leq0,
\end{equation}
where we used the fact that the integral of the first term of
(\ref{eq4}) vanishes, since the flow of the field $v$ preserves the
measure $\mu$.

Integrating by parts the last term of (\ref{eq6}) and applying
Cauchy-Schwartz inequality to the first term of (\ref{eq6}) we get:
\begin{equation}
\label{eq7}
\begin{split}
-2\left(\int(\rho'(H))^2u_t^2d\mu\right)^{\frac{1}{2}}&
\left(\int\rho^2(H)a^2d\mu\right)^{\frac{1}{2}}+\\
+\frac{1}{n}\int\rho^2(H)a^2d\mu -&\int(\rho^2)'(H)(\nabla_q
u)^2d\mu\leq0.
\end{split}
\end{equation}
Notice that (\ref{eq7}) is a quadratic inequality in the quantity
$\left(\int\rho^2(H)a^2d\mu\right)^{\frac{1}{2}}$ Thus the
discriminant $D$ must be non-negative:
$$D=\int(\rho'(H))^2u_t^2d\mu+
\frac{1}{n}\int(\rho^2)'(H)(\nabla_q u)^2d\mu\geq0.
$$
This proves the claim.
\end{proof}
\section{Estimating $D$ from above}
In what follows we shall stretch the function $\rho$ of the previous
section with the help of a small parameter $0<\alpha<1$ in the
following way:
$$
\rho_{\alpha}(x):=\rho(\alpha x).
$$
Then for any $0<\alpha<1$ we have:$$
supp(\rho_{\alpha})=\frac{1}{\alpha}supp(\rho)\subset\left(\frac{1}{\alpha}
(\frac{R^2}{2}+M);+\infty
\right)\subset\left(\frac{R^2}{2}+M;+\infty \right).
$$
Thus theorem \ref{D} applies to every such $\rho_{\alpha}$ and we
have \begin{equation}
\label{al}D_{\alpha}=\int(\rho_{\alpha}'(H))^2u_t^2d\mu+
\frac{1}{n}\int(\rho_{\alpha}^2)'(H)(\nabla_q u)^2d\mu\geq0.
\end{equation}
In this section we prove:
\begin{theorem}
\label{Dalpha} Let $n\geq3$. If $\int(\nabla_q u)^2 dqdt>0$ then
there exists an $\alpha\in(0;1)$ such that $D_{\alpha}<0$.
\end{theorem}
\begin{proof}
Let me denote the first and the second integrals of $D_\alpha$ in
(\ref{al}) by $A$ and $B$ respectively. We need to estimate each of
them from above.

To estimate $A$, use Foubini theorem, then pass to spherical
coordinates $(r=|p|, \omega)$ in the fibers and then to the energy
instead of $|p|$ as follows:
$$
A=\int\left(\int(\rho'_{\alpha}(H))^2r^{n-1}drd\omega\right)u_t^2dqdt=
$$
$$
\omega_n\int\left(\int(\rho'_{\alpha}(H))^2r^{n-2}d(\frac
{r^2}{2}+u)\right)u_t^2dqdt=
$$
$$
\omega_n\int\left(\int(\rho'_{\alpha}(H))^2(2(H-u))^{\frac{n-2}{2}}dH\right)u_t^2dqdt
\leq
$$
$$
\leq \omega_n\int
u_t^2dqdt\left(\int(\rho'_{\alpha}(H))^2(2H)^{\frac{n-2}{2}}dH\right),
$$
where we used $0\leq u$ in the last line of the estimate. In the
last integral we replace $\rho_{\alpha}(H)$ by $\rho(\alpha H)$ and
change the integration variable $H\rightarrow\alpha H$. We have:
\begin{equation}
\begin{split}
\label{A} A\leq \omega_n\int
u_t^2dqdt&\left(\alpha^{\frac{4-n}{2}}\int(\rho'(\alpha
H))^2(2\alpha H)^{\frac{n-2}{2}}d(\alpha H)\right)=\\
&=C_1\alpha^{\frac{4-n}{2}}\omega_n\int u_t^2dqdt,
\end{split}
\end{equation}
where $\omega_n$ is the volume of the unite $(n-1)$-sphere and $C_1$
is the following constant:$$
C_1=\int(\rho'(x))^2(2x)^{\frac{n-2}{2}}dx.
$$
Estimating $B$ we proceed in a similar manner as for $A$:

$$
B=\frac{1}{n}\int(\rho_{\alpha}^2)'(H)(\nabla_q u)^2d\mu=
$$
$$
=\frac{1}{n}\int\left(\int(\rho^2_{\alpha}(H))'r^{n-1}drd\omega\right)|\nabla_q
u|^2dqdt=
$$
$$
=\frac{\omega_n}{n}\int\left(\int(\rho^2_{\alpha}(H))'r^{n-2}d(\frac
{r^2}{2}+u)\right)|\nabla_q u|^2dqdt=
$$
$$
=\frac{\omega_n}{n}\int\left(\int(\rho^2_{\alpha}(H))'(2(H-u))^{\frac{n-2}{2}}dH\right)|\nabla_q
u|^2dqdt.
$$
Integrating by parts in the inner integral we have
\begin{equation}
\label{B1}
B=-\frac{\omega_n(n-2)}{n}\int\left(\int(\rho^2_{\alpha}(H))(2(H-u))^{\frac{n-4}{2}}dH\right)|\nabla_q
u|^2dqdt.
\end{equation}
Notice that the exponent $\frac{n-4}{2}$ in (\ref{B1}) can change
sign therefore we need to split into three cases:

1) Case $n=3$. In this case from (\ref{B1}) we have
$$
B=-\frac{\omega_3}{3}\int\left(\int(\rho^2_{\alpha}(H))(2(H-u))^{-\frac{1}{2}}dH\right)|\nabla_q
u|^2dqdt\leq$$
$$\leq -\frac{\omega_3}{3}\left(\int(\rho^2_{\alpha}(H))(2H)^{-\frac{1}{2}}dH\right)
\int|\nabla_q u|^2dqdt, $$ since $u\geq0$.

 Changing variable $H\rightarrow\alpha H$
in the integral in
 brackets we have:
\begin{equation}
\label {case1} B\leq  -\frac{\omega_3}{3}C_2
\alpha^{-\frac{1}{2}}\int|\nabla_q u|^2dqdt,
\end{equation}
where the constant $C_2$ equals
$$
C_2=\int\rho^2(x)(2x)^{-\frac{1}{2}}dx.
$$

Therefore in this case we have for $D_\alpha$ from
(\ref{A})(\ref{case1}):
$$D_\alpha=A+B\leq \omega_3C_1\alpha^{\frac{1}{2}}\int u_t^2dqdt-\frac{\omega_3}{3}C_2
\alpha^{-\frac{1}{2}}\int|\nabla_q u|^2dqdt.$$

Since $\int|\nabla_q u|^2dqdt>0$ then the right hand side tends to
$-\infty$ as $\alpha$ tends to zero. This proves the theorem for the
first case.

2) Case $n=4$. In this case we compute from (\ref{B1})
\begin{equation}
\label{case2}
B=\frac{\omega_4}{2}\left(\int(\rho^2_{\alpha}(H))dH\right)
\int|\nabla_q u|^2dqdt=\frac{\omega_4}{2}C_2\alpha^{-1}
\int|\nabla_q u|^2dqdt.
\end{equation}
where $C_2=\int\rho^2(x)dx$. So in this case we have for $D_\alpha$
from (\ref{A})(\ref{case2}):
$$
D_\alpha\leq C_1\omega_4\int
u_t^2dqdt-\frac{\omega_4}{2}C_2\alpha^{-1} \int|\nabla_q u|^2dqdt.
$$
In this case again the right hand side tends to $-\infty$ as
$\alpha\rightarrow 0$.

3) Case $n\geq 5$. In this case since $u\leq M$ we have:
$$B=-\frac{\omega_n(n-2)}{n}\int\left(\int(\rho^2_{\alpha}(H))(2(H-u))^{\frac{n-4}{2}}dH\right)|\nabla_q
u|^2dqdt\leq
$$
$$
\leq-\frac{\omega_n(n-2)}{n}\left(\int(\rho^2_{\alpha}(H))(2(H-M))^{\frac{n-4}{2}}dH\right)
\int|\nabla_q u|^2dqdt.
$$
Changing the variable in the integral in brackets
$H\rightarrow\alpha H$ we get:
\begin{equation*}
\begin{split} &B\leq\\
&-\frac{\omega_n(n-2)}{n}\left(\alpha^{\frac{2-n}{2}}\int(\rho^2(\alpha
H)(2(\alpha H-\alpha M))^{\frac{n-4}{2}}d(\alpha H)\right)
\int|\nabla_q u|^2dqdt\leq\\ &
-\frac{\omega_n(n-2)}{n}\left(\alpha^{\frac{2-n}{2}}\int(\rho^2(x)(2(x-
M))^{\frac{n-4}{2}}dx\right) \int|\nabla_q u|^2dqdt,
\end{split}
\end{equation*}
where we used $\alpha M<M$.

Therefore we have
\begin{equation}
\label{case3} B\leq
-\frac{\omega_n(n-2)}{n}C_2\alpha^{\frac{2-n}{2}}\int|\nabla_q
u|^2dqdt,
\end{equation}
where
$$C_2=\int(\rho^2(x)(2(x-
M))^{\frac{n-4}{2}}dx.$$ Thus we have  from (\ref{A})(\ref{case3})
the estimate for $D_\alpha$:
$$D_\alpha=A+B\leq \omega_nC_1\alpha^{\frac{4-n}{2}}\int u_t^2dqdt-
\frac{\omega_n(n-2)}{n}C_2\alpha^{\frac{2-n}{2}}\int|\nabla_q
u|^2dqdt=
$$
$$
=\omega_n\alpha^{\frac{2-n}{2}}\left(C_1\alpha\int u_t^2dqdt-
\frac{(n-2)}{n}C_2\int|\nabla_q u|^2dqdt \right).
$$
Thus also in this case the limit of the right hand side is $-\infty$
when $\alpha\rightarrow 0$. This completes the proof in all the
cases.
\end{proof}
\textbf{Remark.} Notice that the case $n=2$ is excluded in this
method, because for $n=2$ by (\ref{B1}) gives $B=0$ and the argument
breaks down, i.e $D_\alpha$ is indeed non-negative. I don't know if
this is the artifact of the method or the statement of the main
theorem fails in this case.


\end{document}